\documentclass[12pt]{article}

\usepackage{amsfonts,amsthm,latexsym,amsmath,amssymb,amscd,amsmath, mathrsfs, epsf}
\usepackage[enableskew]{youngtab}
\usepackage[utf8]{inputenc}
\usepackage{graphicx}
\usepackage[colorlinks=true,citecolor=black,linkcolor=black,urlcolor=blue]{hyperref}

\theoremstyle{plain}
\newtheorem{theorem}{Theorem}
\newtheorem{lemma}[theorem]{Lemma}

\newtheorem{proposition}[theorem]{Proposition}

\theoremstyle{definition}

\newtheorem{example}[theorem]{Example}

\theoremstyle{remark}
\newtheorem{remark}[theorem]{Remark}

\newcommand{\doi}[1]{\href{http://dx.doi.org/#1}{\texttt{doi:#1}}}

\newcommand{\C}{\mathbb{C}}

\newcommand{\svec}{\mathbf{s}}
\newcommand{\tvec}{\mathbf{t}}

\newcommand{\alphavec}{{\boldsymbol\alpha}}
\newcommand{\betavec}{{\boldsymbol\beta}}
\newcommand{\lambdavec}{{\boldsymbol\lambda}}
\newcommand{\kappavec}{{\boldsymbol\kappa}}
\newcommand{\muvec}{{\boldsymbol\mu}}
\newcommand{\nuvec}{{\boldsymbol\nu}}
\newcommand{\symS}{\mathfrak{S}}

\begin{document}
\title{Schur polynomials, banded Toeplitz matrices and Widom's formula} 

\author{Per Alexandersson\\
\small Department of Mathematics\\[-0.8ex]
\small Stockholm University\\[-0.8ex] 
\small S-10691, Stockholm, Sweden\\
\small\tt per@math.su.se\\
}

\maketitle

\begin{abstract}
We prove that for arbitrary partitions $\lambdavec \subseteq \kappavec,$ and integers $0\leq c<r\leq n,$ the sequence of Schur polynomials
$S_{(\kappavec + k\cdot \mathbf{1}^c)/(\lambdavec + k\cdot \mathbf{1}^r)}(x_1,\dots,x_n)$ for $k$ sufficiently large,
satisfy a linear recurrence. The roots of the characteristic equation are given explicitly.
These recurrences are also valid for certain sequences of minors of banded Toeplitz matrices.

In addition, we show that Widom's determinant formula from 1958 is a special case of a well-known identity for Schur polynomials.

\end{abstract}

\section{Introduction}

\subsection{Minors of banded Toeplitz matrices}
Fix a positive integer $n$ and a finite sequence $s_0,s_1,\dots,s_n$ of complex numbers.
Define an infinite banded Toeplitz matrix $A$ by the formula
\begin{equation}\label{eqn:A}
A:= (s_{j-i}),\, 1\leq i < \infty, 1\leq j < \infty \text{ with } s_i:=0\text{ for } i>n, i<0.
\end{equation}
Given an increasing
$r-$tuple $\alphavec = (\alpha_1,\alpha_2,\dots,\alpha_r)$ and 
an increasing $c-$tuple $\betavec = (\beta_1,\beta_2,\dots,\beta_c)$ of positive integers with $r\leq c \leq n,$
define $D_{\alphavec,\betavec}^k$ as the $k\times k-$matrix obtained by first removing rows indexed by $\{\alpha_i\}_{i=1}^r$ 
and columns indexed by $\{\beta_i\}_{i=1}^c$ from $A$ and then selecting the leading $k \times k-$sub-matrix.
In particular, we let $D_c^k$ to be $D_{\alphavec,\betavec}^k$ for $\alphavec=\emptyset,\betavec=(1,2,\dots,c).$
We will also require $s_0=1$ which is a natural assumption\footnote{
If $s_0=0,$ the first column of $D_{\alphavec,\betavec}^k$ will consist of zeros,
unless $\beta_1=1.$ In the first case, $\det D_{\alphavec,\betavec}^k$ is therefore 0 for every $k>0$ and uninteresting.
In the latter case, we may just as well use the sequence $s_1,s_2,\dots,s_n$
and decrease all entries in $\betavec$ by 1 and obtain the exact same sequence.
Thus, there is no loss of generality if we assume $s_0\neq 0.$
Furthermore, we are interested on the determinants of $D_{\alphavec,\betavec}^k,$
so assuming $s_0:=1$ is not a big restriction and the general case can easily be recovered.}.

A great deal of research has been focused on the asymptotic 
eigenvalue distribution of $D_c^k$ as $k \rightarrow \infty,$ 
the most important are the Szegö limit theorem from 1915,
and the strong Szegö limit theorem from 1952. 

There are many ways to generalize the strong Szegö limit theorem,
for example, the Fisher-Hartwig conjecture from 1968.
Some cases of the conjecture have been promoted to a theorem, 
based on the works of many people the last 20 years, including Widom, Basor, Silberman, Böttcher and Tracy.
A possible refinement of the conjecture is the Basor-Tracy conjecture, see \cite{boettcher2,basor}.

Asymptotics of Toeplitz determinants arises naturally in many areas;
Szegö himself considered the two-dimensional Ising model.
For a more recent application in combinatorics, see \cite{bdj}, 
where the length of the longest increasing subsequence in a random permutation is studied.

A classic result in the theory of banded Toeplitz matrices was obtained by H.~Widom \cite{widom}.
In a modern setting, it may be formulated as follows:

\begin{theorem}(Widom's determinant formula, \cite{boettcher})
Let $\psi(t):= \sum_{i=0}^n s_i t^i.$ 
If the zeros $t_1,t_2,\dots,t_n$ of $\psi(t)=0$ are distinct then, 
for every $k\geq 1,$ 
\begin{equation}\label{eqn:widom}
\det D_c^k = \sum_{\sigma} C_\sigma w_\sigma^k, \quad \sigma \in \binom{[n]}{n-c}
\end{equation}
where
\begin{equation*}
w_\sigma:=(-1)^{n-c} s_n \prod_{i \in \sigma} t_i\text{ and }
C_\sigma:= \prod_{i \in \sigma} t_i^c \prod_{\stackrel{j\in \sigma}{i \notin \sigma}} (t_j - t_i)^{-1}.
\end{equation*}
\end{theorem}

In 1960, by using Widom's formula, P.~Schmidt and F.~Spitzer gave a description of the limit set of the eigenvalues of $D_c^k$
as $k \rightarrow \infty.$ In the above notation, part of their theorem can be stated as follows:
\begin{theorem}(P.~Schmidt, F.~Spitzer, \cite{spitzer})\label{thm:spitzer}
Let $I_k$ denote the $k\times k$-identity matrix and define
\begin{equation*}
B = \left\{ v \Big\vert v = \lim_{k \rightarrow \infty} v_k, \det(D_c^{k} - v_k I_{k})=0 \right\},
\end{equation*}
that is, $B$ is the set of limit points of eigenvalues of $\{D_c^k\}_{k=0}^\infty.$
Let 
\begin{equation*}
f(z) = \sum_{i=0}^n s_i z^{i-c}\text{ and } Q(v,z) = z^c(f(z)-v).
\end{equation*}
Order the moduli of the zeros, $\rho_i(v)$, of $Q(v,z)$ in increasing order, 
$$0 < \rho_1(v) \leq \rho_2(v) \leq \dots \leq \rho_n(v),$$
with possible duplicates counted several times, according to multiplicity.
Let $C = \left\{ v | \rho_c(v) = \rho_{c+1}(v) \right\}.$
Then, $B=C.$
\end{theorem}
The Laurent polynomial $f(z)$ is called the \emph{symbol} associated with the Toeplitz matrix $D_c^k,$
and it is an important tool\footnote{Note that $f$ has a close resemblance with $\psi$ in Widom's formula.} for studying asymptotics.

More recently, a newer approach using the theory of Schur polynomials has been successfully used to
further investigate the series $\{\det D_{\alphavec,\betavec}^k \}_{k=1}^\infty$, e.g. \cite{bumbdiaconis}.
For a recent application of Schur functions in the non-banded case, see \cite{borodin}.

There is also a connection between multivariate orthogonal polynomials and certain determinants of $D_{\alphavec,\betavec}^k,$
considered as functions of $(s_0,s_1,\dots,s_n).$
The solution set to a system of polynomial equations obtained from some $\det D_{\alphavec,\betavec}^k$
converges to the measure of orthogonality as $k \rightarrow \infty.$ 
For example, in 1980, a bivariate generalization of Chebyshev polynomials was constructed by K.~B.~Dunn and R.~Lidl.
Some more recent applications of the theory of symmetric functions
are \cite{beerends,gepner}, where use of Schur polynomials and representation theory 
gives multivariate Chebyshev polynomials.
These multivariate Chebyshev polynomials are also minors of certain Toeplitz matrices.

For example, if $n=2$ and $P_j(s_1,s_2) := \det D_1^j,$
we have that 
$$T_j(x) = P_j(x-\sqrt{x^2-1},x+\sqrt{x^2-1}) = S_{(j)}(x-\sqrt{x^2-1},x+\sqrt{x^2-1}),$$
where $T_j(x)$ is the $j$th Chebyshev polynomial of the second kind, 
and $S_{(j)}$ is the Schur polynomial for the partition with one part of size $j,$ in two variables.

However, the close connection between multivariate Chebyshev polynomials 
and Schur polynomials (and thus minors of banded Toeplitz matrices)
has not yet been sufficiently investigated.

\subsection{Main results}
We start with giving a Schur polynomial interpretation of $\det D_{\alphavec,\betavec}^k.$

Set $s_i := s_i(x_1,x_2,\dots,x_n)$ where $s_i$ is the $i:$th elementary symmetric polynomial.
We impose a natural\footnote{This ensures that no leading matrix of $D_{\alphavec,\betavec}^k$ is upper-triangular 
with a zero on the main diagonal, which would force
$\det D_{\alphavec,\betavec}^k$ to vanish.} restriction on $\alphavec$ and $\betavec,$ namely $\alpha_i\geq \beta_i$ for $i=1,2,\dots,r.$

\begin{proposition}
In the above notation, for $k$ sufficiently large, we have
\begin{equation}\label{eqn:minorisschur}
\det D_{\alphavec,\betavec}^k  = S_{(\lambdavec + k\muvec)/(\kappavec+k\nuvec)}(x_1,x_2,\dots,x_n), 
\end{equation}
where $S_{(\lambdavec + k\muvec)/(\kappavec+k\nuvec)}$ is a skew Schur polynomial defined below.
Here $\lambdavec, \kappavec, \muvec, \nuvec$ are partitions given by
\begin{equation*}
\lambdavec=(1-\beta_1,2-\beta_2,\dots,c-\beta_c), \quad \kappavec =(1-\alpha_1,2-\alpha_2,\dots,r-\alpha_r)
\end{equation*}
\begin{equation*}
\muvec = (\underbrace{1,1,\dots,1}_c), \quad \nuvec = (\underbrace{1,1,\dots,1}_r).
\end{equation*}
\end{proposition}

The conditions on $\alphavec$ and $\betavec$ ensure that $S_{(\lambdavec + k\muvec)/(\kappavec+k\nuvec)}$
is well-defined for $k\geq \max(\alpha_r-r,\beta_c-c).$
(Identity \eqref{eqn:minorisschur} is proven below in Prop.~\ref{prop:matrixisschur}, 
a similar identity is proven in \cite{bumbdiaconis}.)

To state our main first result, we need to define the following. 
Set $b:=\binom{n}{c-r}$ and define the finite sequence of polynomials $\{Q_i(x_1,\dots,x_n)\}_{i=0}^{b}$ by the identity
\begin{equation}\label{eq:characteristicequation}
\sum_{k=0}^{b} Q_{b-k} t^k = 
\prod_{\stackrel{\sigma \subseteq [n]}{|\sigma| = c-r }} (t-x_{\sigma_1} x_{\sigma_2}\cdots x_{\sigma_{c-r}}).
\end{equation}

\begin{theorem}\label{thm:rootproducts}
Given strictly increasing sequences $\alphavec, \betavec$ of positive integers of length $r$ resp. $c$ with $c\leq r,$ satisfying 
$\alpha_i\geq \beta_i$ for $i=1,2,\dots,r,$ we have  
\begin{equation}\label{eqn:charequation}
\sum_{k=0}^{b} Q_{b-k} \det(D_{\alphavec,\betavec}^{k+j}) =0 \quad \text{ for all } j\geq \max(\alpha_c-c,\beta_r-r).
\end{equation}
\end{theorem}
\noindent(Here, we use the convention that the determinant of an empty matrix is 1.)

\begin{remark}
For the case $D_c^k$, the existence of recurrence \eqref{eqn:charequation} was previously shown in \cite[Thm.~2]{pascallike}, 
but its length and coefficients were not given explicitly.
Also, Theorem \ref{thm:rootproducts} has close resemblance to a result given in \cite[Thm. 5.1]{hou}.
It is however unclear whether \cite{hou} implies Theorem \ref{thm:rootproducts}.
Additionally, in contrast to \cite{hou}, our proof of Thm.~\ref{thm:rootproducts} is short and purely combinatorial.
\end{remark}

To formulate the second result, define
\begin{equation}\label{eqn:chi1}
\chi(t) = \prod_{i=1}^n (t-x_i) = (-1)^{n} \sum_{i=0}^n  (-t)^{n-i} s_{i}(x_1,x_2,\dots,x_n). 
\end{equation}
We then have the following theorem, which is equivalent to Widom's formula:
\begin{theorem}\label{lemma:widomequiv}(Modified Widom's formula) 

If the zeros $x_1,x_2,\dots,x_n$ of $\chi(t)=0$ are distinct then, for every $k\geq 1,$ 
\begin{equation*}
\det D_c^k = \sum_{\tau} 
\prod_{i \in \tau} x_i^k
\prod_{ \stackrel{i \in \tau}{j \notin \tau} } \frac{x_i}{x_i-x_j},
\quad \tau \in \binom{[n]}{c}
\end{equation*}
\end{theorem}

\begin{remark}
Below, we show that this (and therefore Widom's original formula) 
follows immediately from a known identity for the Hall polynomials.

Note that Theorem \ref{thm:rootproducts} can be verified easily using Widom's original formula.
I was informed that there is an unpublished result by S.~Delvaux and A.~L.~García which uses a Widom-type formula 
for block Toeplitz matrices to give recurrences similar to \eqref{eqn:charequation}.
\end{remark}

\section{Preliminaries}

Given two integer partitions $\lambdavec = (\lambda_1, \lambda_2,\dots,\lambda_n),$
$\muvec = (\mu_1, \mu_2,\dots,\mu_n)$ with $\lambda_1 \geq \lambda_2 \geq \dots \geq \lambda_n \geq 0,$
$\mu_1 \geq \mu_2 \geq \dots \geq \mu_n \geq 0,$
we say that $\lambdavec \supseteq \muvec$ if $\lambda_j\geq \mu_j$ for $j=1\dots n.$
Given two such partitions, one constructs the associated 
\emph{skew Young diagram}\footnote{In the case $\muvec = (0,0,\dots,0),$ the word \emph{skew} is to be omitted.} by having 
$n$ left-adjusted rows of boxes, where row $j$ contains $\lambda_j$ boxes,
and then removing the first $\mu_j$ boxes from row $j.$ The removed boxes is called the \emph{skew part} of the tableau.
\begin{example}
The following diagram is obtained from the partitions $(4,2,1)$ and $(2,2),$
and it is said to be of the \emph{shape} $(4,2,1)/(2,2)$:
$$\young(\blacksquare\blacksquare\hfil\hfil,\blacksquare\blacksquare,\hfil)$$
(We will omit/add trailing zeros in partitions when the intended length is known from the context.)
\end{example}
The \emph{conjugate} of a partition is the partition obtained by transposing the corresponding tableau.
For example, the conjugate of $(4,2,1)/(2,2)$ is $(3,2,1,1)/(2,2).$

Given such a diagram, a \emph{(skew) semi-standard Young tableau} (we will use just the word tableau from now on) is an 
assignment of positive integers to the boxes, such that each row is weakly increasing, and each column is strictly increasing.

We define the \emph{(skew) Schur polynomial} $S_{\lambdavec/\muvec}(x_1,x_2,\dots,x_n)$ as
$$S_{\lambdavec/\muvec}(x_1,x_2,\dots,x_n) = \sum x_1^{h_1} \cdots x_n^{h_n}$$
where the sum is taken over all tableaux of shape $\lambdavec/\muvec,$
and $h_j$ counts the number of boxes containing $j$ for each particular tableau.
No box may contain an integer greater than $n.$
When $\muvec = (0,0,\dots,0)$ we just write $S_\lambdavec.$ 
To clarify, each Schur polynomial is associated with a Young diagram, and each monomial in such polynomial
corresponds to a set of tableaux. We use this correspondence extensively.
For example, the tableau above yields the Schur polynomial
$$x_1^3+x_2^3+x_3^3 + 2( x_1^2 x_2 + x_1^2 x_3 + x_2^2 x_1 + x_2^2 x_3 + x_3^2 x_1 + x_3^2 x_2) + 3x_1 x_2 x_3.$$

The following formula express the (skew) Schur polynomials in a determinant form:
\begin{proposition}\label{prop:jacobitrudi}(Jacobi-Trudi identity \cite{macdonald})

Let $\lambdavec \supseteq \muvec$ be partitions with at most $n$ parts and let $\lambdavec',\muvec'$ be their conjugate partitions (with at most $k$ parts).
Then the (skew) Schur polynomial $S_{\lambdavec/\muvec}$ is given by
\begin{equation*}
S_{\lambdavec/\muvec}(x_1, x_2, \dots, x_n) = \begin{vmatrix}
s_{\lambda'_1-\mu'_1} & s_{\lambda'_1-\mu'_1+1} & \dots & s_{\lambda'_1-\mu'_1+k-1} \\
s_{\lambda'_2-\mu'_2-1} & s_{\lambda'_2-\mu'_2} & \dots & s_{\lambda'_2-\mu'_2+k-2}\\
\vdots &  &  \ddots & \vdots \\
s_{\lambda'_k-\mu'_k-k+1} &\dots && s_{\lambda'_k-\mu'_k} \\
\end{vmatrix}
\end{equation*}
where $s_j := s_j(x_1,\dots,x_n),$ the elementary symmetric functions in $x_1,\dots,x_n.$
Here, $s_{j}\equiv 0$ for $j<0.$
\end{proposition}
It is clear that every (skew) Schur polynomial $S_{\lambdavec/\muvec}(x_1,\dots,x_n)$
is symmetric in $x_1,\dots,x_n.$

\section{Proofs}
The following proposition shows that certain minors of banded Toeplitz 
matrices may be interpreted as Schur polynomials.
\begin{proposition}\label{prop:matrixisschur}
Let $D_{\alphavec,\betavec}^k$ be defined as above. 
Then,
$$\det D_{\alphavec,\betavec}^k = S_{(\lambdavec+k \muvec)/(\kappavec + k \nuvec)}(x_1,x_2,\dots,x_n)$$
where 
\begin{equation*}
\lambdavec=(1-\beta_1,2-\beta_2,\dots,c-\beta_c), \quad \kappavec =(1-\alpha_1,2-\alpha_2,\dots,r-\alpha_r)
\end{equation*}
and
\begin{equation*}
\muvec = (\underbrace{1,1,\dots,1}_c), \quad \nuvec = (\underbrace{1,1,\dots,1}_r).
\end{equation*}
\end{proposition}
\begin{proof}
Consider the matrix $A$ defined in \eqref{eqn:A}, where the indices (of $s$) on the main diagonal are all 0.
Now, removing the rows $\alphavec$ will \emph{decrease} the index on row $i$
by $\#\{j|\alpha_j-j+1\leq i\}.$
Similarly, removing the columns $\betavec$ will \emph{increase}
the index in column $i$ by $\#\{j|\beta_j-j+1\leq i\}.$
\emph{After} removing rows and columns, the diagonal of the resulting matrix, $\tilde{A}$, is given by
\begin{equation*}
\left( \#\{j|\beta_j-j+1\leq i\} - \#\{j|\alpha_j-j+1\leq i\} \right)_{i=1}^\infty.
\end{equation*}
Now, the leading $k \times k$ minor of $\tilde{A}$ is $D_{\alphavec,\betavec}^k$ and its anti-diagonal transpose
has the same determinant as $D_{\alphavec,\betavec}^k.$ The main diagonal in the anti-diagonal transposed matrix equals
\begin{equation}
\begin{aligned}\label{eqn:matrixdiag}
&\left( \#\{j|\beta_j-j+1\leq k-i+1\} - \#\{j|\alpha_j-j+1\leq k-i+1\} \right)_{i=1}^k = & \\
&\left( \#\{j|\beta_j\leq k+j-i\} - \#\{j|\alpha_j\leq k+j-i\} \right)_{i=1}^k &
\end{aligned}
\end{equation}

Now, well-known properties of partition conjugation imply
\begin{equation*}
\begin{split}
(\lambdavec + k \muvec)'= (\#\{j| k+j-\beta_j\geq 1\},\#\{j| k+j-\beta_j\geq 2\},\dots, \#\{j| k+j-\beta_j\geq k\}), \\
(\kappavec + k \nuvec)'= (\#\{j| k+j-\alpha_j\geq 1\},\#\{j| k+j-\alpha_j\geq 2\},\dots, \#\{j| k+j-\alpha_j\geq k\}).
\end{split}
\end{equation*}
Rewriting this we obtain
\begin{equation*}
(\lambdavec + k \muvec)'= (\#\{j| \beta_j \leq k+j-i\})_{i=1}^k, \quad (\kappavec + k \nuvec)'= (\#\{j| \alpha_j \leq k+j-i \})_{i=1}^k.
\end{equation*}
Finally, using $(\kappavec + k \nuvec)/(\lambdavec + k \muvec)$ in the Jacobi-Trudy identity, Prop.~\ref{prop:jacobitrudi},
yields a $k \times k-$matrix with diagonal entries
\begin{equation*}
(\lambdavec + k \muvec)'-(\kappavec + k \nuvec)' =  (\#\{j| \beta_j \leq k+j-i\} - \#\{j| \alpha_j \leq k+j-i \})_{i=1}^k.
\end{equation*}
This expression coincides with the expression for $\det D_{\alphavec,\betavec}^k$ in \eqref{eqn:matrixdiag},
and now it is straightforward to see that all other matrix entries coincides as well.
\end{proof}

\subsection{Young tableaux and sequence insertion}

To prove Theorem \ref{thm:rootproducts}, we need to define a new combinatorial operation on semi-standard skew Young tableaux.
Namely, given a tableau $T$ with $n$ rows, we define an insertion of a sequence $\tvec = t_1<t_2<\dots<t_c$
into $T$ as follows. Each $t_i$ is inserted into row $i,$
such that the resulting row is still weakly increasing. (Clearly, there is a unique way to do this.)
If there is no row $i,$ we create a new left-adjusted row consisting of one box which contains $t_i.$
We call this operation \emph{sequence insertion} of $\tvec$ into $T.$

\begin{lemma}
The result of sequence insertion is a tableau. 
\end{lemma}
\begin{proof}
It is clear that it suffices to check that the resulting columns are strictly increasing.
Furthermore, it suffices to show that any two boxes in adjacent rows are strictly increasing.
Let us consider rows $i$ and $i+1$ after inserting $t_i$ and $t_{i+1},$ $t_i<t_{j+1}.$ 
There are three cases to consider:

\textbf{Case 1:}
The numbers $t_i$ and $t_{i+1}$ are in the same column:
$$
\begin{bmatrix}
\cdots & a_1 & t_i &  a_2 & \cdots & a_m & \cdots \\
\cdots & b_1 & t_{i+1} & b_2 & \cdots & b_{m} & \cdots \\
\end{bmatrix}
$$
Since $t_i<t_{i+1},$ and all the other columns are unchanged, the columns are strictly increasing.

\textbf{Case 2:}
The number $t_i$ is to the right of $t_{i+1}$:
$$
\begin{bmatrix}
\cdots & t_i & a_1 & a_2 & \cdots & a_{m-1} & a_m & \cdots \\
\cdots & b_1 & b_2 & b_3 & \cdots & b_{m}   & t_{i+1} & \cdots \\
\end{bmatrix}
$$
The columns where strictly increasing before the insertion.
Therefore, $t_i\leq a_1 < b_1,$ $a_m < b_m \leq t_{i+1}$ and $a_j<b_j\leq b_{j+1}.$
It follows that all the columns are strictly increasing.

\textbf{Case 3:}
The number $t_i$ to the left of $t_{i+1}$:
$$
\begin{bmatrix}
\cdots & a_1 & a_2 & \cdots & a_{m-1} & a_m & t_i & \cdots \\
\cdots & t_{i+1} & b_1 & b_2 & b_3 & \cdots & b_{m} & \cdots \\
\end{bmatrix}
$$
We have that $a_j\leq t_i < t_{i+1}\leq b_k$ for $1\leq j,k\leq m,$ since the rows are increasing.
Thus, it is clear that all the columns are strictly increasing. 
It is easy to see that the result is a tableau even if $c\neq n.$
\end{proof}
Notice that different sequence insertions commute, i.e., inserting sequence $\svec$ into $T$ followed by $\tvec$,
yields the same result as the reverse order of insertion.

We may extend the notion of sequence insertion to skew tableaux as follows: 
First put negative integers in the skew part, such that the negative integers in each particular row have the same value,
and the columns are strictly increasing.
The result is a regular tableau, (but with some negative entries), so we may perform sequence insertion.
The negative entries still form a skew part of the tableau, and we may remove these to obtain a skew tableau.

Note that we may also allow negative entries in a sequence, which after insertion, are removed. 
The result is a skew tableau. The following example illustrates this:

\begin{example}
Here, we insert the sequence $(-1,2,3)$ into a skew tableau of shape $(4,3,3,2)/(2,1,1):$
$$
\Yvcentermath1
\young(\blacksquare\blacksquare11,\blacksquare12,\blacksquare34,14)
\rightarrow 
\young(\blacksquare\blacksquare\blacksquare11,\blacksquare122,\blacksquare334,14)
$$
\end{example}

\begin{lemma}\label{lem:insertion} 
Let $S_{\lambdavec/\muvec}(x_1,\dots,x_n)$ be a (skew) Schur polynomial. Then, for any $k \geq 0,$
the coefficient of $x_1^{h_1} \dots x_n^{h_n}$ in $x_{t_1} \dots x_{t_c} S_{\lambdavec/\muvec}$
with $0<t_1<t_2<\dots<t_c$ counts the number of ways to insert the sequence $(-k,\dots,-2,-1,t_1,t_2,\dots,t_c)$ into a 
skew tableau of shape $\lambdavec/\muvec$ such that the resulting (skew) tableau has exactly $h_i$ boxes with value $i.$
\end{lemma}
\begin{proof}
Since there is exactly one way to insert a sequence into a (skew) tableau, the equality is clear. 
\end{proof}
Expressing Schur polynomials and products of the form $x_{t_1} \dots x_{t_c} S_{\lambdavec/\muvec}$
as a sum of \emph{monic} monomials, we have a 1-1-correspondence between a monic monomials and tableaux. 
Thus, in what follows, we may sloppily identify these two objects when proving Theorem \ref{thm:rootproducts}:
\begin{proof}[Proof of Theorem \ref{thm:rootproducts}]
We may assume that $\alpha_i\geq \beta_i$ for $i=1,\dots,r.$ Otherwise, all determinants vanish,
and the identity is trivially true. With these assumptions we may use the Schur polynomial interpretation.

Let $b:=\binom{n}{c-r}$ and let $j \geq \max(r-\alpha_r, c-\beta_c).$ 
Rewriting \eqref{eq:characteristicequation} using identity \eqref{eqn:minorisschur} yields 
\begin{equation}\label{eq:schuridentity2}
S_{(\lambdavec + (b+j)\muvec)/(\kappavec + (b+j)\nuvec)} = \sum_{k=0}^{b-1} Q_{b-k} S_{(\lambdavec + (k+j)\muvec)/(\kappavec + (k+j)\nuvec)}.
\end{equation}

Now, notice that the difference between tableaux of shape $(\lambdavec + k\muvec)/(\kappavec + k\nuvec)$
and tableaux of shape $(\lambdavec + (k-1)\muvec)/(\kappavec + (k-1)\nuvec)$ is that the former
contains an extra column of the form 
$$\underbrace{\blacksquare,\dots,\blacksquare}_r \underbrace{\square,\dots,\square}_{c-r}.$$
Therefore, each tableau of shape $(\lambdavec + k\muvec)/(\kappavec + k\nuvec)$, ($k > \max(r-\alpha_r, c-\beta_c)$)
may be obtained from some tableau of shape $(\lambdavec + (k-1)\muvec)/(\kappavec + (k-1)\nuvec)$ by inserting
a sequence of the form 
$$(-r,\dots,-1,t_1,t_2,\dots,t_{c-r}).$$
Together with Lemma \ref{lem:insertion}, this implies that all 
tableaux\footnote{monic monomials} in $S_{(\lambdavec + (b+j)\muvec)/(\kappavec + (b+j)\nuvec)}$ 
are also tableaux\footnotemark[\value{footnote}] in
\begin{equation}\label{eq:step21}
Q_{1} S_{(\lambdavec + (b+j-1)\muvec)/(\kappavec + (b+j-1)\nuvec)}.
\end{equation}
Hence, there is almost an equality between $S_{(\lambdavec + (b+j)\muvec)/(\kappavec + (b+j)\nuvec)}$ and \eqref{eq:step21}, 
but some tableaux in $S_{(\lambdavec + (b+j)\muvec)/(\kappavec + (b+j)\nuvec)}$ may be obtained by using 
different sequence insertions.
Those tableaux are exactly the tableaux that may be obtained by using 
$$S_{(\lambdavec + (b+j-2)\muvec)/(\kappavec + (b+j-2)\nuvec)}$$ 
using two \emph{different} sequence insertions.

Thus, $S_{(\lambdavec + (b+j)\muvec)/(\kappavec + (b+j)\nuvec)}$ is almost given by
$$Q_{1} S_{(\lambdavec + (b+j-1)\muvec)/(\kappavec + (b+j-1)\nuvec)} +  Q_{2} S_{(\lambdavec + (b+j-2)\muvec)/(\kappavec + (b+j-2)\nuvec)}.$$
(Multiplying with $Q_{2}$ can be viewed as performing all possible pairs of two different sequence insertions, and then there is a sign.)

Repeating this reasoning using inclusion/exclusion yields \eqref{eq:schuridentity2}.
\end{proof}

\begin{remark}
Note that the technical condition $j \geq \max(\alpha_r-r,\beta_c-c)$ in \eqref{eqn:charequation} is indeed necessary. 
For example, with $n=2$, $\{\det D_{(),(2)}^k\}_{k=0}^2$ do not satisfy the recurrence but $\{\det D_{(),(2)}^k\}_{k=1}^3$ do:
$$x_1 x_2 \cdot 1 - (x_1 + x_2) \begin{vmatrix} 1 \end{vmatrix} + 1 \begin{vmatrix} 1 & x_1 x_2 \\ 0 & 1 \end{vmatrix} \neq 0$$
but
$$x_1 x_2 \cdot \begin{vmatrix} 1 \end{vmatrix} -
(x_1 + x_2) \begin{vmatrix} 1 & x_1 x_2 \\ 0 & 1 \end{vmatrix} + 
1 \begin{vmatrix} 1 & x_1  x_2 & 0\\ 0 & x_1 + x_2 & x_1  x_2 \\ 0 & 1 & x_1 + x_2 \end{vmatrix} = 0.$$

This circumstance is a clear distinction of our result to the result in \cite{hou}, 
where the corresponding recurrence (for a slightly different type of objects) does not need such additional restriction.
\end{remark}

\subsection{Widom's formula}

We will now show that Theorem \ref{lemma:widomequiv} is equivalent to Widom's formula.
\begin{lemma}
Theorem \ref{lemma:widomequiv} is equivalent to Widom's formula \eqref{eqn:widom}.
\end{lemma}
\begin{proof}
It is clear from \eqref{eqn:chi1} that $(-t)^n \psi(-1/t)=\chi(t),$ so the roots of these polynomials are related by $t_i = -1/x_i.$
Substituting $t_i \mapsto -1/x_i$
in \eqref{eqn:widom} and canceling signs yields
\begin{equation*}\label{eqn:bkwrec}
\det D_c^k = \sum_{\sigma} 
\left( \frac{s_n}{ \prod_{i \in \sigma} x_i } \right)^k 
\left(\prod_{i \in \sigma} x_i^{-c} \right) \prod_{\stackrel{j\in \sigma}{i \notin \sigma}} \left(\frac{1}{x_j} - \frac{1}{x_i}\right)^{-1}.
\end{equation*}
Using that $s_n = x_1 x_2 \cdots x_n$ and rewriting the last product, we get
\begin{equation*}
\det D_c^k = \sum_{\sigma} 
 \prod_{i \notin \sigma} x_i ^k 
\left(\prod_{i \in \sigma} x_i^{-c} \right) 
\prod_{\stackrel{j\in \sigma}{i \notin \sigma}} x_j \left(\frac{x_i}{x_i-x_j}\right).
\end{equation*}
Now notice that the last product produces $x_j^{c}$, since $|[n]\setminus \sigma| = c.$ 
Thus, we may cancel these with the middle product. Finally, putting $\tau = [n]\setminus \sigma$ yields the desired identity.
\end{proof}

Thus, to prove Widom's formula, it suffices to prove Theorem \ref{lemma:widomequiv}.
However, it is a direct consequence of the following identity:

\begin{proposition}\label{prop:macdonald}(Identity for Hall polynomials, \cite[p.~104, eqn.~(2.2)]{macdonald})

The Schur polynomial $S_\lambda(x_1,\dots,x_n)$ satisfy
\begin{equation*}
S_\lambda(x_1,\dots,x_n) = 
\sum_{w \in \symS_n/\symS_n^\lambda} w\left( x_1^{\lambda_1}\dots x_n^{\lambda_n} \prod_{\lambda_i > \lambda_j} \frac{x_i}{x_i-x_j} \right) 
\end{equation*}
where $\symS_n^\lambda$ is the subgroup of permutations with the property that $\lambda_{w(j)}=\lambda_j$ for $j=1\dots n,$
and $w$ acts on the indices of the variables.
\end{proposition}

\begin{proof}[Proof of Thm.~\ref{lemma:widomequiv}]
Let $\lambdavec = (k,\dots,k,0,\dots,0)$ with $c$ entries equal to $k.$
Then $\symS_n^\lambda$ is the subgroup consisting of permutations, permuting the first $c$ variables,
and the last $n-c$ variables independently. 
The condition $\lambda_i > \lambda_j$ will only be satisfied if $\lambda_i=k$ and $\lambda_j=0.$ 
Therefore Prop.~\ref{prop:macdonald} immediately implies Theorem \ref{lemma:widomequiv}.
\end{proof}

\subsection{Applications}
Theorem \ref{thm:rootproducts} can be used to give a shorter 
proof a result of Schmidt and Spitzer in \cite{spitzer}, by using the main result in \cite{bkw},
which reads as follows:

Let $\{P_n(z)\}$ be a sequence of polynomials satisfying
\begin{equation}
P_{n+b} = -\sum_{j=1}^b q_j(z) P_{n+b-j}(z), 
\end{equation}
where the $q_j$ are polynomials. The number $x \in \C$ is a limit of zeros of $\{P_n\}$
if there is a sequence of $z_n$ s.t. $P_n(z_n)=0$ and $\lim_{n\rightarrow \infty} z_n = x.$

For fixed $z,$ we have roots $v_i, 1\leq i \leq b$ of the characteristic equation
$$v^b + \sum_{j=1}^b q_j(z) v^{b-j} = 0.$$

For any $z$ such that the $v_i(z)$ are distinct, we may express $P_n(z)$ as follows:
\begin{equation}\label{eqn:bkwwidom}
P_n(z) = \sum_{j=1}^b r_j(z) v_i(z)^n.
\end{equation}

Under the non-degeneracy conditions that $\{P_n\}$ do not satisfy a recurrence of length less than $b,$
and that there is no $w$ with $|w|=1$ such that $v_i(z) = w v_j(z)$ for some $i \neq j,$
the following holds:
\begin{theorem}\label{thm:bkw}(See \cite{bkw}).
Suppose $\{P_n\}$ satisfy \eqref{eqn:bkwrec}. Then $x$ is a limit of zeros if and only if the roots $v_i$ of the
characteristic equation can be numbered so that one of the following is satisfied:
\begin{enumerate}
\item $|v_1(x)|>|v_j(x)|, 2\leq j \leq b$ and $r_1(z)=0.$
\item $|v_1(x)| = |v_2(x)|=\dots = |v_l(x)| > |v_{j}(x)|, l+1\leq j \leq b$
for some $l\geq 2.$
\end{enumerate}
\end{theorem}

We are now ready to prove a generalization of Thm.~\ref{thm:spitzer}:
\begin{theorem}
Fix natural numbers $n$ and $0<c<n.$
Let $\gamma_1,\gamma_2,\dots,\gamma_d$ be a sequence of $d$ integers such that $c<\gamma_1<\gamma_2<\dots<\gamma_d.$
Set $\alphavec = (\gamma_1,\dots,\gamma_d)$ and set $\betavec = (1,2,\dots,c,\gamma_1,\dots,\gamma_d).$
Define
\begin{equation*}
B = \left\{ v | v = \lim_{k \rightarrow \infty} v_k, \det(D_{\alphavec,\betavec}^{k} - v_k I_{k})=0 \right\}.
\end{equation*}
Let 
\begin{equation*}
f(z) = \sum_{i=0}^n s_i z^{i-c},\quad Q(v,z) = z^c(f(z)-v).
\end{equation*}
Order the moduli of the zeros, $\rho_i(v)$, of $Q(v,z)$ in increasing order, 
with possible duplicates counted several times, according to multiplicity:
$$0 < \rho_1(v) \leq \rho_2(v) \leq \dots \leq \rho_n(v).$$
Let $C = \left\{ v | \rho_{c}(v) = \rho_{c+1}(v) \right\}.$
Then, $B=C\cup W$ where $W \subset \C$ is a finite set of points.
\end{theorem}
\begin{proof}
Consider the sequence of matrices $\{D_{\alphavec,\betavec}^k - v I_k \}_{k=K}^\infty,$ $K = \gamma_d-d$.
It is easy to see that the main diagonal of all these matrices will be of the form $s_c-v,$ 
and no other entries involve either $s_c$ or $v.$
Now, define $s'_i(v) = s_i - \delta_{ic}v,$ where $\delta_{ij}$ is the Dirac delta.
Let us modify \eqref{eqn:chi1} and define
\begin{equation}
\chi(v,t) = \prod_{i=1}^n (t-x_i(v)) = (-1)^{n}\sum_{i=0}^n (-t)^{n-i} s'_{i}(v).
\end{equation}
Notice that $\chi(v,t) = (-1)^nQ(v,-1/t).$ If we enumerate the roots of $\chi(v,t)$ according to their magnitude,
$$0<|x_1(v)|\leq |x_2(v)| \leq \dots \leq |x_n(v)|,$$
we have that $|x_i(v)| = 1/\rho_{i}(v)$ for  $1\leq i \leq n.$

From Thm.~\ref{thm:rootproducts} it follows that the series $\{D_{\alphavec,\betavec}^k - v I_k \}_{k=K}^\infty$
satisfy the characteristic equation
\begin{equation}\label{eqn:characteristicequation2}
\prod_{\stackrel{\sigma \subseteq [n]}{|\sigma| = c }} (t-x_{\sigma_1}(v) x_{\sigma_2}(v)\cdots x_{\sigma_{c}}(v))=0. 
\end{equation}

It is evident that for this characteristic equation the non-degeneracy conditions hold.
All roots are different, and we require all of them for the equation to be symmetric under permutation of the $x_i,$ hence, the recurrence is minimal. 
The second condition holds since the left-hand side of the characteristic equation is irreducible, see \cite{bkw}
for details.

From Thm.~\ref{thm:bkw}, it follows
that the zeros of $\det(D_{\alphavec,\betavec}^{k_m} - v_m I_{k_m})=0$
accumulate exactly where two or more of the largest zeros of \eqref{eqn:characteristicequation2} coincide in magnitude,
or when the corresponding $r_j(z)=0$ in \eqref{eqn:bkwwidom}. 
The latter case\footnote{Alternative 1 cannot be satisfied if $d=0,$ equation \eqref{eqn:bkwwidom} is then Widom's formula, 
and all coefficients $r_i(z)$ are non-zero since all roots $x_j(v)$ are nonzero.} can only hold for only a finite number of points.
The first case is satisfied exactly when 
\begin{align*}
|x_{n-c-1}(v) x_{n-c+1}(v) x_{n-c+2}(v) \cdots x_{n}(v)| &= 
|x_{n-c}(v) x_{n-c+1}(v) x_{n-c+2}(v) \cdots x_{n}(v)| \\
&\Leftrightarrow& \\  
|x_{n-c-1}(v)| &= |x_{n-c}(v)| \\
&\Leftrightarrow&\\
\rho_c(v) &= \rho_{c+1}(c).
\end{align*}
This concludes the proof.
\end{proof}
The same strategy as above may be used to find limits of \emph{generalized eigenvalues},
as defined in \cite{duitsthesis}.

It is also possible to generalize Thm.~\ref{thm:rootproducts} to more general sequences of skew Schur polynomials,
$\{S_{(\kappavec + k\nuvec)/(\lambdavec + k \nuvec)}\}_{k=0}^\infty$ for $\nuvec \subseteq \muvec.$
This may be used to find asymptotics for the number of skew tableaux of certain shapes,
and asymptotics for the set of zeros of the Schur polynomials.

\subsection*{Acknowledgement} I would like to thank my advisor, B.~Shapiro, for introducing me to this subject,
and A.~Kuijlaars for the reference to Widom's formula and the hospitality during my visit to Katholieke Universiteit Leuven.
Also, many thanks to S.~Alsaadi, J.~Backelin, M.~Duits, M.~Leander and M.~Tater for helpful discussions.

\end{document}